\documentclass{amsart}

\usepackage{amsmath,amsthm,amssymb,amsfonts,bbm,amscd,mathrsfs}
\usepackage{tikz,tikz-cd}
\usepackage{url,graphicx}
\usepackage[utf8]{inputenc}
\usepackage[pagebackref,colorlinks]{hyperref}
\usepackage{enumitem}
\usepackage{cleveref}

\newtheorem{thm}{Theorem}[section]
\newtheorem{lem}[thm]{Lemma}

\newtheorem{cor}[thm]{Corollary}

\numberwithin{equation}{section}

\theoremstyle{definition}
\newtheorem{rmk}[thm]{Remark}

\newtheorem{defn}[thm]{Definition} 
\newtheorem{example}[thm]{Example}
\newtheorem{question}[thm]{Question}

\newtheorem{prop-defn}[thm]{Proposition-Definition}

\title{S-equivalence for algebraic stacks}

\author{Xucheng Zhang}
\date{}
\address{Yau Mathematical Sciences Center \\
Tsinghua University \\
Beijing 100084 \\
China \\
Email: \href{mailto: zhangxucheng@mail.tsinghua.edu.cn}{zhangxucheng@mail.tsinghua.edu.cn}}

\begin{document}

\begin{abstract}
We generalize the notion of S-equivalence, previously defined for semistable vector bundles, to points in arbitrary algebraic stacks and use it to describe the identification of points when passing to the moduli space. As applications, we recover the classical S-equivalence by considering the identification of semistable vector bundles in the moduli space, and discuss its relation to separatedness of the moduli space.
\end{abstract}

\maketitle


\section{Introduction}
This paper fits into our long-term goal of trying to
\begin{enumerate}
\item[(S1)]
find more (proper) moduli spaces of vector bundles over a curve and, when such a moduli space exists,
\item[(S2)]
exploring which points get identified in the moduli space.
\end{enumerate}
Using Geometric Invariant Theory (GIT) there exists a projective variety parameterizing S-equivalence classes of semistable vector bundles over a curve. Perhaps surprisingly, if we want a moduli space of vector bundles containing unstable objects, GIT won't be of much help as we prove in \cite{dario-zhang} that the semistable locus is the maximal open locus to which GIT can be applied to construct a moduli space. Fortunately, there is an alternative approach to (S1) in this situation: the existence criterion of moduli spaces for algebraic stacks \cite{MR4665776}. Using this result we find in \cite{dario-zhang} separated (non-proper) moduli spaces of vector bundles containing unstable objects. Contrary to the GIT construction, in the latter approach we only know from the general fact (\cite[Theorem 5.3.1 (4)]{MR3272912}) that two vector bundles get identified in the moduli spaces if and only if their orbit closures intersect. Our original motivation is to study when orbit closures of vector bundles intersect.

In this paper we show for a large class of algebraic stacks (including all algebraic stacks of characteristic $0$ and quotient stacks), the condition ``orbit closures intersect'' can be translated into a more readily verifiable one. In particular, this gives a complete answer to (S2). Two guiding examples we have in mind are stack of semistable vector bundles over a curve and quotient stacks. 
\begin{itemize}
\item 
In the first example, two points are identified in the moduli space if and only if they are S-equivalent, and in this case both of them degenerate to the unique closed point (a polystable vector bundle) in their S-equivalence class. Using the Rees construction, any degeneration of vector bundles can be realized as a morphism from $[\mathbf{A}^1/\mathbf{G}_m]$.
\item 
In the second example, two points are identified in the moduli space if and only if their orbit closures intersect, and in this case both of them degenerate to the unique closed point (a fixed point) in their intersection. Using the Hilbert-Mumford criterion, any degeneration of points can be realized as a morphism from $[\mathbf{A}^1/\mathbf{G}_m]$.
\end{itemize}
This leads us to consider morphisms from $[\mathbf{A}^1/\mathbf{G}_m]$, which are called ``filtrations'' in \cite{DHL2014} or ``very close degenerations'' in \cite{MR3758902}, and they give rise to the expected analogue of S-equivalence for points in arbitrary algebraic stacks, referred to as ``quasi-S-equivalence'' (Definition \ref{def:quasi-S}). Furthermore, if an algebraic stack is locally reductive (in the sense of \cite[Definition 2.5]{MR4665776}) we can also define ``quasi-Jordan-H\"{o}lder filtrations'' (Definition \ref{defn:JH}) for its points, which plays the same role for quasi-S-equivalence as Jordan-H\"{o}lder filtration for S-equivalence. The identification of points in the moduli space can be described using either quasi-S-equivalence or quasi-Jordan-H\"{o}lder filtrations.
\begin{thm}[Theorem \ref{thm:equivalence}]\label{thm:in-equivalence}
Let $\mathscr{X}$ be a locally reductive algebraic stack, locally of finite type over a field $k$. Suppose $\mathscr{X}$ admits an adequate moduli space $\phi: \mathscr{X} \to X$. Then for any geometric points $x_1,x_2 \in \mathscr{X}(\kappa)$, the following are equivalent:
\begin{itemize}
\item 
The points $x_1,x_2$ are quasi-S-equivalent.
\item 
There exist quasi-Jordan-H\"{o}lder filtrations witnessing the quasi-S-equivalence of $x_1,x_2$.
\item 
The equality $\phi(x_1)=\phi(x_2)$ holds.
\end{itemize}
In particular, quasi-S-equivalence is an equivalence relation on points of $\mathscr{X}$.
\end{thm}
\begin{rmk}
The technical condition ``local reductivity'' emerges due to the lack of a local structure theorem of algebraic stacks in positive characteristic, which requires that every point of the stack specializes to a closed point and \'{e}tale locally around any closed point the stack is of the form $[\mathrm{Spec}(A)/\mathrm{GL}_N]$ for some $N$. 
For an algebraic stack $\mathscr{X}$ admitting an adequate moduli space, local reductivity is automatic in the following cases:
\begin{itemize}
\item 
The adequate moduli space of $\mathscr{X}$ is a good moduli space (e.g. if $\mathscr{X}$ is of characteristic $0$). 

Indeed, in this case every closed point of $\mathscr{X}$ has linearly reductive stabilizer by \cite[Theorem 4.1]{MR4665776} and we can apply the local structure theorem \cite[Theorem 1.1]{Alper2019-1} to conclude, as observed in \cite[Remark 2.6]{MR4665776}.
\item
The stack $\mathscr{X}$ admits a representable morphism to a quotient stack for an action of a reductive group 
(e.g. if $\mathscr{X}$ is a quasi-compact open substack of the stack of coherent sheaves over a curve).

Indeed, this is covered in the proof of \cite[Proposition 3.28]{dario-zhang}.
\end{itemize}
\end{rmk}
For the stack of coherent sheaves over a curve, morphisms from $[\mathbf{A}^1/\mathbf{G}_m]$ can be described in terms of genuine filtrations via the Rees construction. Therefore in this case Theorem \ref{thm:in-equivalence} has the following simple form.
\begin{cor}[Corollary \ref{cor:main}]
Let $C$ be a smooth projective connected curve over an algebraically closed field $k$. Denote by $\mathscr{C}oh_n$ the stack of rank $n$ coherent sheaves over $C$. Let $\mathscr{U} \subseteq \mathscr{C}oh_n$ be an open substack that admits an adequate moduli space. Then two geometric points $\mathcal{E}_1,\mathcal{E}_2 \in \mathscr{U}(\kappa)$ are identified in the adequate moduli space if and only if there exist filtrations $\mathcal{E}_1^\bullet,\mathcal{E}_2^\bullet$ of $\mathcal{E}_1,\mathcal{E}_2$ respectively such that $\mathrm{gr}(\mathcal{E}_1^\bullet)=\mathrm{gr}(\mathcal{E}_2^\bullet) \in \mathscr{U}(\kappa)$.
\end{cor}
In particular, the identification of semistable vector bundles in the moduli space gives rise to the notion of S-equivalence.
\begin{cor}[Corollary \ref{cor:usual-S}]\label{In-cor:usual-S}
Two semistable vector bundles are identified in the adequate moduli space if and only if they are S-equivalent.
\end{cor}
\begin{rmk}
In the GIT approach a priori we need to know the ``correct'' equivalence relation on the parametrized objects to get a nice representable moduli functor. This is how the S-equivalence of semistable vector bundles appears as a way to eliminate the non-separatedness. If we apply the existence criterion developed in \cite{MR4665776} to the stack $\mathscr{B}un_n^{ss}$ of rank $n$ semistable vector bundles, we again obtain a proper moduli space (see, e.g. \cite{MR4480534}) and Corollary \ref{In-cor:usual-S} implies that the S-equivalence appears naturally if we consider the identification of semistable vector bundles in the moduli space. In turn this gives the correct equivalence relation pertaining to the original moduli functor.
\end{rmk}
The identification of points in an algebraic stack naturally relates to the separatedness of its moduli space (which is equivalent to the S-completeness of the stack by \cite[Proposition 3.48 (2)]{MR4665776}), as the special fibers of two families with isomorphic generic fibers must be identified if we want a separated moduli space. Such special fibers are therefore called non-separated (Definition \ref{defn:non-sep}). For an S-complete algebraic stack we want to describe its non-separated points and characterize when its open substacks are also S-complete. To this end we introduce another condition called axis-completeness (Definition \ref{defn:axis-comp}), which might be easier to check than S-completeness in practice.
\begin{thm}[Theorem \ref{thm:equivalence-non-opp} and Theorem \ref{thm:open-S-comp}]
Let $\mathscr{X}$ be a smooth S-complete algebraic stack with affine stabilizers over a field $k$. Then
\begin{itemize}
\item 
Two geometric points in $\mathscr{X}$ are non-separated if and only if they have opposite filtrations (Definition \ref{defn:opp-fil}).
\item
An open substack $\mathscr{U} \subseteq \mathscr{X}$ is S-complete if and only if the open immersion $\mathscr{U} \hookrightarrow \mathscr{X}$ is axis-complete.
\end{itemize}
\end{thm}
For the stack of coherent sheaves over a curve, this recovers \cite[Corollary 3.23 and Proposition 3.17]{dario-zhang}.
\begin{cor}[Corollary \ref{cor:non=opp-in-Coh} and Corollary \ref{cor:nonsep=id}]
Let $C$ be a smooth projective connected curve over an algebraically closed field $k$. Denote by $\mathscr{C}oh_n$ the stack of rank $n$ coherent sheaves over $C$. Then
\begin{itemize}
\item 
Two geometric points $\mathcal{E},\mathcal{E}' \in \mathscr{C}oh_n$ are non-separated if and only if they have opposite filtrations $\mathcal{E}^\bullet,\mathcal{E}'_\bullet$ (Example \ref{eg:opp-in-Coh}).
\item 
An open substack $\mathscr{U} \subseteq \mathscr{X}$ is S-complete if and only if for any geometric points $\mathcal{E},\mathcal{E}' \in \mathscr{U}(\kappa)$ with opposite filtrations $\mathcal{E}^\bullet,\mathcal{E}'_\bullet$, we have $\mathrm{gr}(\mathcal{E}^\bullet)=\mathrm{gr}(\mathcal{E}'_\bullet) \in \mathscr{U}(\kappa)$.
\end{itemize}
\end{cor}
\subsection*{Notation}
Let $C$ be a smooth projective connected curve over an algebraically closed field $k$. Denote by $\mathscr{C}oh_n$ (resp., $\mathscr{B}un_n$) the stack of rank $n$ coherent sheaves (resp., vector bundles) over $C$.
\subsection*{Acknowledgements}
I am grateful to Jianping Wang, Xueqing Wen for helpful discussions, and Dario Wei\ss mann for giving detailed comments on earlier versions of the paper.
\section{Preliminaries}
\subsection{S-equivalence}
As we have already seen, morphisms from $[\mathbf{A}^1/\mathbf{G}_m]$ are of great importance and we record it into a definition.
\begin{defn}[Filtration, \cite{DHL2014} or \cite{MR3758902}]
Let $\mathscr{X}$ be an algebraic stack over a field $k$ and $x \in \mathscr{X}(\kappa)$ be a geometric point. A \emph{filtration} of $x$ is a morphism $f: [\mathbf{A}^1_{\kappa}/\mathbf{G}_{m,\kappa}] \to \mathscr{X}$ such that $f(1) \cong x$. It is called \emph{non-trivial} if $f(0) \ncong x$.
For any filtration $f: [\mathbf{A}^1_{\kappa}/\mathbf{G}_{m,\kappa}] \to \mathscr{X}$ of $x$, the point $f(0) \in \mathscr{X}(\kappa)$ is called the \emph{associated graded point} of the filtration $f$. 
\end{defn}
\begin{example}\label{eg:filtration}
In several cases filtrations can be explicitly written down.
\begin{enumerate}
\item 
(Rees construction, \cite[Lemma 1.10]{MR3758902}) Let $\mathscr{U} \subseteq \mathscr{C}oh_n$ be an open substack. Then a morphism $[\mathbf{A}^1/\mathbf{G}_{m}] \to \mathscr{U}$ is the same as a data of a rank $n$ coherent sheaf $\mathcal{E}$ over $C$ together with a finite $\mathbf{Z}$-graded filtration $\mathcal{E}^\bullet$ such that the associated graded sheaf $\mathrm{gr}(\mathcal{E}^\bullet) \in \mathscr{U}$.
\item
(Hilbert-Mumford criterion, \cite[Lemma 1.6]{MR3758902}) Let $\mathscr{X}=[X/G]$ be the quotient stack for an action of a smooth affine algebraic group $G$ on a quasi-separated algebraic space $X$ of finite type, both defined over a field $k$. Then a morphism $[\mathbf{A}^1/\mathbf{G}_{m}] \to \mathscr{X}$ is the same as a data of a point $x \in X$ and a cocharacter $\lambda: \mathbf{G}_m \to G$ such that $\lim_{t \to 0} \lambda(t).x$ exists.
\end{enumerate}
\end{example}
Recall that in order to define the S-equivalence of semistable vector bundles we first single out a distinguished class of filtrations, called Jordan-H\"{o}lder filtrations, and then identify those bundles with the same associated graded sheaves. Using the Rees construction this can be interpreted as identifying those points with ``special'' filtrations such that the images of $0$ are the same. For arbitrary stacks a priori it is not clear which kind of filtrations should be special. Nevertheless we can capture the underlying idea of S-equivalence and use filtrations alone. To distinguish it from the usual S-equivalence we refer to it as ``quasi-S-equivalence''.
\begin{defn}[Quasi-S-equivalence]\label{def:quasi-S}
Let $\mathscr{X}$ be an algebraic stack over a field $k$. Two geometric points $x_1,x_2 \in \mathscr{X}(\kappa)$ are called \emph{quasi-S-equivalent} in $\mathscr{X}$ if there exist filtrations $f_1,f_2$ of $x_1,x_2$ respectively such that $f_1(0)=f_2(0)$.
\end{defn}
The first observation is that any morphism to algebraic spaces must be constant on quasi-S-equivalent points.
\begin{lem}\label{lem:const-on-qS}
Let $\mathscr{X}$ be an algebraic stack over a field $k$ and $x_1,x_2 \in \mathscr{X}(\kappa)$ be two quasi-S-equivalent geometric points. Then they have the same image under any morphism $f: \mathscr{X} \to W$ to an algebraic space $W$.
\end{lem}
The converse in \Cref{lem:const-on-qS} is not true as all geometric points, quasi-S-equivalent or not, have the same image under the structure morphism $\mathscr{X} \to \mathrm{Spec}(k)$.
\begin{proof}
Let $f_1,f_2$ be filtrations of $x_1,x_2$ respectively such that $f_1(0)=f_2(0)$. Then the composition $[\mathbf{A}^1_{\kappa}/\mathbf{G}_{m,\kappa}] \xrightarrow{f_i} \mathscr{X} \xrightarrow{f} W$ is constant as it factors uniquely through the good moduli space $[\mathbf{A}^1_\kappa/\mathbf{G}_{m,\kappa}] \to \mathrm{Spec}(\kappa)$, so
\[
f(x_1)=f \circ f_1(1)=f \circ f_1(0)=f \circ f_2(0)=f \circ f_2(1)=f(x_2).
\]
\end{proof}
Example \ref{eg:filtration} (1) then implies
\begin{lem}\label{lem:S-in-Coh}
Let $\mathscr{U} \subseteq \mathscr{C}oh_n$ be an open substack. Two geometric points $\mathcal{E}_1,\mathcal{E}_2 \in \mathscr{U}(\kappa)$ are quasi-S-equivalent if and only if there exist filtrations $\mathcal{E}_1^\bullet,\mathcal{E}_2^\bullet$ of $\mathcal{E}_1,\mathcal{E}_2$ respectively such that $\mathrm{gr}(\mathcal{E}_1^\bullet)=\mathrm{gr}(\mathcal{E}_2^\bullet) \in \mathscr{U}(\kappa)$. In particular, quasi-S-equivalent points in $\mathscr{U}$ have the same determinant.
\end{lem}
For quasi-S-equivalent vector bundles, there is a simple characterization.
\begin{lem}\label{lem:S-in-Bun}
Two geometric points $\mathcal{E}_1,\mathcal{E}_2 \in \mathscr{B}un_n(\kappa)$ are quasi-S-equivalent if and only if $\det(\mathcal{E}_1)=\det(\mathcal{E}_2)$. In particular
\begin{enumerate}
\item 
Quasi-S-equivalence is an equivalence relation on points of $\mathscr{B}un_n$.
\item 
The determinant morphism $\det: \mathscr{B}un_n \to \mathrm{Pic}$ is a categorical moduli space, i.e., any morphism $\mathscr{B}un_n \to W$ to an algebraic space $W$ factors uniquely through the morphism $\det$.
\end{enumerate}
\end{lem}
\begin{proof}
The only if part is proved in Lemma \ref{lem:S-in-Coh}. For the if part, choose an integer $m \gg 0$ such that both $\mathcal{E}_1(m)$ and $\mathcal{E}_2(m)$ are globally generated. Then \cite[Proposition 2.6 (ix)]{MR3013028} gives a short exact sequence
\begin{gather*}
0 \to \mathcal{O}_C^{\oplus n-1} \to \mathcal{E}_i(m) \to \det(\mathcal{E}_i(m)) \to 0, \text{ i.e.,} \\
0 \to \mathcal{O}_C^{\oplus n-1}(-m) \to \mathcal{E}_i \to \det(\mathcal{E}_i(m))(-m) \to 0.
\end{gather*}
Since $\det(\mathcal{E}_i(m))=\det(\mathcal{E}_i)(nm)$, we are done by Lemma \ref{lem:S-in-Coh}. In particular, the part (1) is clear. 

By Lemma \ref{lem:const-on-qS}, any morphism $f: \mathscr{B}un_n \to W$ to an algebraic space $W$ is constant on each fiber of $\det$. Since $\det$ is surjective, such factorization $\mathrm{Pic} \to W$ is necessarily unique. To show the existence, note that the morphism $\det: \mathscr{B}un_n \to \mathrm{Pic}$ has a section $s: \mathrm{Pic} \to \mathscr{B}un_n$ given by $\mathcal{L} \mapsto \mathcal{L} \oplus \mathcal{O}_C^{\oplus n-1}$, any morphism $f: \mathscr{B}un_n \to W$ thus factors through $f \circ s: \mathrm{Pic} \to W$. This proves the part (2).
\end{proof}
For semistable vector bundles, the notion of quasi-S-equivalence recovers that of S-equivalence.
\begin{lem}\label{lem:qS-recover}
Two geometric points $\mathcal{E}_1,\mathcal{E}_2 \in \mathscr{B}un_n^{ss}(\kappa)$ are quasi-S-equivalent if and only if they are S-equivalent.
\end{lem}
\begin{proof}
This is proved using Lemma \ref{lem:S-in-Coh}. By definition, S-equivalence implies quasi-S-equivalence. Conversely, if $\mathcal{E}_1,\mathcal{E}_2 \in \mathscr{B}un_n^{ss}(\kappa)$ are two quasi-S-equivalent geometric points, then there exist filtrations $\mathcal{E}_1^\bullet,\mathcal{E}_2^\bullet$ of $\mathcal{E}_1,\mathcal{E}_2$ respectively such that $\mathrm{gr}(\mathcal{E}_1^\bullet)=\mathrm{gr}(\mathcal{E}_2^\bullet) \in \mathscr{B}un_n^{ss}(\kappa)$. The filtration $\mathcal{E}_i^\bullet$ can be refined to be a Jordan-H\"{o}lder filtration $\mathcal{E}_i^\text{JH}$ of $\mathcal{E}_i$ and $\mathrm{gr}(\mathcal{E}_1^\text{JH})=\mathrm{gr}(\mathcal{E}_2^\text{JH})$ by the following lemma. This means $\mathcal{E}_1,\mathcal{E}_2$ are S-equivalent.
\end{proof}
\begin{lem}
Let $\mathcal{E}^\bullet$ be a filtration of a semistable vector bundle $\mathcal{E}$ such that $\mathrm{gr}(\mathcal{E}^\bullet)$ is also semistable. Then the filtration $\mathcal{E}^\bullet$ can be refined to be a Jordan-H\"{o}lder filtration $\mathcal{E}^\mathrm{JH}$ of $\mathcal{E}$. Furthermore, $\mathrm{gr}(\mathcal{E}^\mathrm{JH})$ depends only on $\mathrm{gr}(\mathcal{E}^\bullet)$.
\end{lem}
\begin{proof}
Since the associated graded bundle $\mathrm{gr}(\mathcal{E}^\bullet)$ is semistable, each successive quotient $\mathcal{E}^i/\mathcal{E}^{i-1}$ is semistable of the same slope. By induction we reduce to the following easy claim and the fact that stable components in a Jordan-H\"{o}lder filtration of any semistable vector bundle are unique.

\vspace{0.3em}

\noindent \textsc{Claim.}
Let $\mathcal{G}_1 \subseteq \mathcal{G}_2$ and $\mathcal{F}_1 \subseteq \mathcal{G}_2/\mathcal{G}_1$ be subsheaves. Then there exists a unique intermediate sheaf $\mathcal{G}_1 \subseteq \mathcal{F} \subseteq \mathcal{G}_2$ such that
\[
\mathcal{F}/\mathcal{G}_1 \cong \mathcal{F}_1 \text{ and } \mathcal{G}_2/\mathcal{F} \cong (\mathcal{G}_2/\mathcal{G}_1)/\mathcal{F}_1.
\]
Indeed, the sheaf $\mathcal{F} \subseteq \mathcal{G}_2$ is the preimage of the subsheaf $\mathcal{F}_1 \subseteq \mathcal{G}_2/\mathcal{G}_1$ under the quotient map $\mathcal{G}_2 \to \mathcal{G}_2/\mathcal{G}_1$. 
\end{proof}
Example \ref{eg:filtration} (2) then implies
\begin{lem}\label{lem:S-in-quotient}
Let $\mathscr{X}=[X/G]$ be the quotient stack for an action of a smooth affine algebraic group $G$ on a quasi-separated algebraic space $X$ of finite type, both defined over a field $k$. Two geometric points $\overline{x}_1,\overline{x}_2 \in \mathscr{X}(\kappa)$ are quasi-S-equivalent if and only if $\overline{G.x_1} \cap \overline{G.x_2} \neq \emptyset$, where $x_1,x_2 \in X(\kappa)$ are some lifts of $\overline{x}_1,\overline{x}_2 \in \mathscr{X}(\kappa)$ respectively.
\end{lem}
\begin{rmk}
There are some comments on quasi-S-equivalence.
\begin{itemize}
\item 
This is not an interesting notion for points of algebraic spaces.

In this case any morphism from $[\mathbf{A}^1/\mathbf{G}_m]$ factors through the good moduli space $[\mathbf{A}^1/\mathbf{G}_m] \to \mathrm{Spec}(k)$ so any filtration must be trivial.
\item 
In general, quasi-S-equivalence is not an equivalence relation on points of an algebraic stack. 

It is reflexive, symmetric but not necessarily transitive. For example in the quotient stack $[\mathbf{P}^1/\mathbf{G}_m]$, the point $1$ is quasi-S-equivalent to both $0$ and $\infty$, but $0$ and $\infty$ are not quasi-S-equivalent, by Lemma \ref{lem:S-in-quotient}.

However, if an algebraic stack admits an adequate moduli space, then quasi-S-equivalence is an equivalence relation on points of the stack, see Theorem \ref{thm:equivalence}. The converse is not always true, e.g. quasi-S-equivalence is an equivalence relation on points of $\mathscr{B}un_n$ (see Lemma \ref{lem:S-in-Bun}) but $\mathscr{B}un_n$ does not admit an adequate moduli space.
\item
Quasi-S-equivalence depends on the ambient stack. 

Indeed, let $\mathscr{U} \subseteq \mathscr{X}$ be a substack and $x_1,x_2 \in \mathscr{U}(\kappa)$ be two geometric points. If $x_1,x_2$ are quasi-S-equivalent in $\mathscr{U}$, then they are quasi-S-equivalent in $\mathscr{X}$. The converse is not always true. For example, all geometric points of $[\mathbf{A}^2/\mathbf{G}_m]$, where $\mathbf{G}_m$ acts on $\mathbf{A}^2$ with weight $1,1$, are quasi-S-equivalent by Lemma \ref{lem:S-in-quotient}. 
This is never the case in the open substack $[\mathbf{A}^2-\{0\}/\mathbf{G}_m] \cong \mathbf{P}^1$, where no geometric points are quasi-S-equivalent.
\end{itemize}
\end{rmk}
\subsection{Jordan-H\"{o}lder filtration}
The key feature of Jordan-H\"{o}lder filtrations of semistable vector bundles is that their associated graded bundles are the unique closed point (in the ambient stack of semistable vector bundles) that the original bundles can specialize to. This motivates the following definition of quasi-Jordan-H\"{o}lder filtrations for points of arbitrary algebraic stacks.
\begin{defn}[Quasi-Jordan-H\"{o}lder filtration]\label{defn:JH}
Let $\mathscr{X}$ be an algebraic stack over a field $k$. For any geometric point $x \in \mathscr{X}(\kappa)$, a filtration $f: [\mathbf{A}^1_{\kappa}/\mathbf{G}_{m,\kappa}] \to \mathscr{X}$ of $x$ is said to be \emph{quasi-Jordan-H\"{o}lder} if
\begin{itemize}
\item 
the closure $\overline{\{x\}}$ of $x$ contains a unique closed point, say $x_0 \in \mathscr{X}(\kappa)$.
\item 
$f(0)=x_0$.
\end{itemize}
\end{defn}
\begin{rmk}
A quasi-Jordan-H\"{o}lder filtration of a geometric point
\begin{itemize}
\item 
does not always exist. 

On one hand, the closure of a point may contain several closed points, e.g. the point $1$ in $[\mathbf{P}^1/\mathbf{G}_m]$. On the other hand, unlike schemes, where any point of a locally noetherian scheme specializes to a closed point (see \cite[\href{https://stacks.math.columbia.edu/tag/02IL}{Tag 02IL}]{stacks-project}), even algebraic stacks that are locally of finite type (e.g. $\mathscr{B}un_n$) may have no closed point.
\item 
need not be unique. 

However, the associated graded point is unique. This agrees with the classical situation.
\end{itemize}
\end{rmk}
To ensure the existence of quasi-Jordan-H\"{o}lder filtrations for all points, we put the following two conditions pertaining to the stack in question:
\begin{itemize}
\item 
The stack admits an adequate moduli space. 

This guarantees that the closure of any point contains a unique closed point, see Lemma \ref{lem:unique-closed}. However, if the closure of any point in $\mathscr{X}$ contains a unique closed point, then $\mathscr{X}$ does not necessarily admit an adequate moduli space. For example, let $\mathbf{G}_m$ act on $\mathbf{P}^1$ via $t.[x:y]=[tx:y]$ so $0$ and $\infty$ are the only closed orbits. Let $C$ be the curve obtained from $\mathbf{P}^1$ by glueing $0$ and $\infty$, so any orbit closure in $C$ contains a unique closed orbit, but $C \to \mathrm{Spec}(k)$ is not a good quotient. 

\item 
The stack is locally reductive. 

This guarantees that any specialization to a closed point can be realized as a morphism from $[\mathbf{A}^1/\mathbf{G}_m]$ (see \cite[Lemma 3.24]{MR4665776}).
\end{itemize}
\begin{lem}\label{lem:unique-closed}
Let $\mathscr{X}$ be a locally reductive algebraic stack, locally of finite type over a field $k$. Suppose either
\begin{enumerate}
\item 
The stack $\mathscr{X}$ is $\Theta$-reductive, or
\item
The stack $\mathscr{X}$ admits an adequate moduli space $\phi: \mathscr{X} \to X$.
\end{enumerate}
Then for any geometric point $x \in \mathscr{X}(\kappa)$, there exists a unique closed point $x_0 \in \mathscr{X}(\kappa)$ in its closure.
\end{lem}
\begin{rmk}
There are comments about Lemma \ref{lem:unique-closed}.
\begin{itemize}
\item 
The point $x_0$ is the unique closed point in $\mathscr{X}$ that $x$ can specialize to. If $\mathscr{X}$ admits an adequate moduli space $\phi: \mathscr{X} \to X$, it is also the unique closed point in the fiber $\phi^{-1}(\phi(x)) \subseteq \mathscr{X}$. 
\item
The condition (2) implies (1) if $\mathscr{X}$ has separated diagonal and the adequate moduli space $X$ is locally noetherian, see \cite[Proposition 3.21 (3)]{MR4665776}.
\end{itemize}
\end{rmk}
\begin{proof}
The conclusion under assumption (1) is \cite[Lemma 3.25]{MR4665776}. Suppose $\mathscr{X}$ admits an adequate moduli space $\phi: \mathscr{X} \to X$. Since $\mathscr{X}$ is locally reductive, any (geometric) point $x \in \mathscr{X}(\kappa)$ specializes to a closed point. This implies that the closure $\overline{\{x\}}$ contains at least one closed points, say $x_0 \in \mathscr{X}(\kappa)$. If $x'_0 \in \mathscr{X}(\kappa)$ is another closed point in $\overline{\{x\}}$, then $\phi(x_0) \neq \phi(x'_0)$ by \cite[Theorem 5.3.1 (4)]{MR3272912}, which is impossible since by \cite[Theorem 5.3.1 (5)]{MR3272912} we have $\phi(\overline{\{x\}})=\phi(x)$.
\end{proof}
\begin{lem}[\cite{MR4665776}, Lemma 3.24]\label{lem:fromTheta}
Let $\mathscr{X}$ be a locally reductive algebraic stack, locally of finite type over a field $k$. Then any specialization of a geometric point to a closed point in $\mathscr{X}$ can be realized as a morphism $[\mathbf{A}^1/\mathbf{G}_m] \to \mathscr{X}$.
\end{lem}
Putting all these together, we see that for a locally reductive algebraic stack admitting an adequate moduli space quasi-Jordan-H\"{o}lder filtrations exist for any of its points. 
\begin{cor}
Let $\mathscr{X}$ be a locally reductive algebraic stack, locally of finite type over a field $k$. Suppose $\mathscr{X}$ admits an adequate moduli space. Then quasi-Jordan-H\"{o}lder filtrations exist for any geometric point of $\mathscr{X}$.
\end{cor}
For semistable vector bundles, the notion of quasi-Jordan-H\"{o}lder filtration recovers that of Jordan-H\"{o}lder filtration. For this we first need to identify all closed points in $\mathscr{B}un_n^{ss}$.
\begin{lem}\label{lem:closed=polystable}
A geometric point in $\mathscr{B}un_n^{ss}$ is closed if and only if it is polystable.
\end{lem}
\begin{proof}
For any geometric point $\mathcal{E} \in \mathscr{B}un_n^{ss}(\kappa)$, let $\mathcal{E}=\mathcal{E}_1^{\oplus n_1} \oplus \cdots \oplus \mathcal{E}_s^{\oplus n_s}$ be the decomposition into distinct indecomposables (see \cite[Theorem 3]{MR86358}). Then each $\mathcal{E}_i$ is again semistable of the same slope. 

If $\mathcal{E}$ is not polystable, i.e., the summand $\mathcal{E}_i$ is strictly semistable for some $i$, then there exists a non-trivial proper stable subbundle $0 \neq \mathcal{F}_i \subsetneq \mathcal{E}_i$ of the same slope and thus the quotient $\mathcal{E}_i/\mathcal{F}_i$ is also semistable of the same slope. Consider the following filtration of $\mathcal{E}$
\[
\mathcal{E}^\bullet: 0 \subseteq \mathcal{F}_i \subseteq \mathcal{E}_i \subseteq \mathcal{E}_i^{\oplus n_i} \subseteq \mathcal{E}.
\]
The associated graded sheaf $\mathrm{gr}(\mathcal{E}^\bullet)$ is semistable and not isomorphic to $\mathcal{E}$. By the Rees construction $\mathrm{gr}(\mathcal{E}^\bullet) \in \overline{\{\mathcal{E}\}}$, i.e., $\mathcal{E} \in \mathscr{B}un_n^{ss}(\kappa)$ is not closed.

Conversely, if $\mathcal{E}$ is polystable, then by Lemma \ref{lem:unique-closed} there exists a unique closed point in $\overline{\{\mathcal{E}\}}$, say $\mathcal{E}' \in \overline{\{\mathcal{E}\}}$. By Lemma \ref{lem:fromTheta} and the Rees construction there exists a filtration $\mathcal{E}^\bullet$ of $\mathcal{E}$ such that $\mathcal{E}'=\mathrm{gr}(\mathcal{E}^\bullet) \in \mathscr{B}un_n^{ss}(\kappa)$. Then each term $\mathcal{E}^i$ in the filtration $\mathcal{E}^\bullet$ is a subbundle of $\mathcal{E}$ of the same slope, which has to be a direct summand of $\mathcal{E}$ (see, e.g. \cite[Corollary 1.6.11]{MR2665168}). This shows that $\mathcal{E}' \cong \mathcal{E}$ and thus $\mathcal{E} \in \mathscr{B}un_n^{ss}(\kappa)$ is closed.
\end{proof}
\begin{lem}
For any geometric point $\mathcal{E} \in \mathscr{B}un_n^{ss}(\kappa)$, a filtration of $\mathcal{E}$ is a quasi-Jordan-H\"{o}lder filtration if and only if it corresponds to a Jordan-H\"{o}lder filtration of $\mathcal{E}$.
\end{lem}
\begin{proof}
For any filtration $\mathcal{E}^\bullet$ of $\mathcal{E}$, the associated graded sheaf $\mathrm{gr}(\mathcal{E}^\bullet)$ is closed if and only if it is polystable (Lemma \ref{lem:closed=polystable}), if and only if $\mathcal{E}^\bullet$ is a Jordan-H\"{o}lder filtration.   
\end{proof}
\begin{rmk}
One can prove the similar results for the stack of principal bundles over a curve, using the description of morphisms from $[\mathbf{A}^1/\mathbf{G}_m]$ established in \cite[Lemma 1.13]{MR3758902}.
\end{rmk}
\section{Identification of points}
After these preparations we can relate identification of points in the moduli space to quasi-S-equivalence, as well as to quasi-Jordan-H\"{o}lder filtrations.
\begin{thm}\label{thm:equivalence}
Let $\mathscr{X}$ be a locally reductive algebraic stack, locally of finite type over a field $k$. Suppose $\mathscr{X}$ admits an adequate moduli space $\phi: \mathscr{X} \to X$. Then for any geometric points $x_1,x_2 \in \mathscr{X}(\kappa)$, the following are equivalent:
\begin{enumerate}
\item 
The equality $\phi(x_1)=\phi(x_2)$ holds.
\item 
The points $x_1,x_2$ are quasi-S-equivalent.
\item 
There exist quasi-Jordan-H\"{o}lder filtrations $f_1,f_2$ of $x_1,x_2$ respectively such that $f_1(0)=f_2(0)$.
\end{enumerate}
In particular, quasi-S-equivalence is an equivalence relation on points of $\mathscr{X}$.
\end{thm}
\begin{rmk}
This confirms the very expectation that quasi-S-equivalent points should be identified in the moduli space (whenever it exists). Unlike S-equivalence, quasi-S-equivalence just tells us which points get identified in the moduli space, it has nothing to do with the separatedness of the moduli space.
\end{rmk}
\begin{proof}
The implication (3) $\Rightarrow$ (2) is clear. The implication (2) $\Rightarrow$ (1) follows from \cite[Theorem 5.3.1 (5)]{MR3272912} since $\overline{\{x_1\}} \cap \overline{\{x_2\}}$ contains the point $f_1(0)=f_2(0)$ for any filtrations $f_1,f_2$ of $x_1,x_2$ witnessing their quasi-S-equivalence. Finally, if $\phi(x_1)=\phi(x_2)$, then $x_{1,0}=x_{2,0}$ as they are both the unique closed point in the fiber of $\phi(x_1)=\phi(x_2)$ under the adequate moduli space $\phi: \mathscr{X} \to X$. Since $\mathscr{X}$ is locally reductive, by Lemma \ref{lem:fromTheta} any specialization $x_i \leadsto x_{i,0}$ can be realized as a morphism $f_i: [\mathbf{A}^1_{\kappa}/\mathbf{G}_{m,\kappa}] \to \mathscr{X}$. This shows (1) $\Rightarrow$ (3).
\end{proof}
For coherent sheaves, Lemma \ref{lem:S-in-Coh} and Theorem \ref{thm:equivalence} yield the following:
\begin{cor}\label{cor:main}
Let $\mathscr{U} \subseteq \mathscr{C}oh_n$ be an open substack. Then for any geometric points $\mathcal{E}_1,\mathcal{E}_2 \in \mathscr{U}(\kappa)$, the following are equivalent:
\begin{enumerate}
\item 
The points $\mathcal{E}_1,\mathcal{E}_2$ are quasi-S-equivalent.
\item 
There exist filtrations $\mathcal{E}_1^\bullet,\mathcal{E}_2^\bullet$ of $\mathcal{E}_1,\mathcal{E}_2$ respectively such that $\mathrm{gr}(\mathcal{E}_1^\bullet)=\mathrm{gr}(\mathcal{E}_2^\bullet) \in \mathscr{U}(\kappa)$.
\end{enumerate}
Moreover, if $\mathscr{U}$ admits an adequate moduli space $\phi: \mathscr{U} \to U$, then (1) and (2) are further equivalent to:
\begin{enumerate}
\item[(3)]
The points $\mathcal{E}_1,\mathcal{E}_2$ are identified in $U$.
\end{enumerate}
\end{cor}
\begin{rmk}
There are some comments on the condition $(2)$ in Corollary \ref{cor:main}.
\begin{itemize}
\item 
It is not enough to just assume in (2) that the associated graded sheaves are the same. There exist open substacks $\mathscr{U} \subseteq \mathscr{C}oh_n$ and geometric points $\mathcal{E}_1,\mathcal{E}_2 \in \mathscr{U}(\kappa)$ with filtrations $\mathcal{E}_1^\bullet,\mathcal{E}_2^\bullet$ respectively such that $\mathrm{gr}(\mathcal{E}_1^\bullet)=\mathrm{gr}(\mathcal{E}_2^\bullet) \notin \mathscr{U}(\kappa)$, so are not identified in the adequate moduli space.
\begin{itemize}
\item
(Non-separated example) In the open substack $\mathscr{B}un_n^{simple}$ of simple rank $n$ vector bundles, no points are identified in the good moduli space. However, there are points with opposite filtrations (see \cite[Corollary 3.19]{dario-zhang}).
\item
(Proper example) In the open substack $\mathscr{B}un_2^{s}$ of stable rank $2$ vector bundles, no points are identified in the good moduli space. However, any non-split extension
\[
0 \to \mathcal{L}_0 \to \mathcal{E} \to \mathcal{L}_1 \to 0
\]
is stable and defines a point of $\mathscr{B}un_2^{s}$, where $\mathcal{L}_i$ is a degree $i$ line bundle.
\end{itemize}
\item
It is incorrect to only consider opposite filtrations (see \cite[Definition 3.15]{dario-zhang}) in (2). There exist open substacks $\mathscr{U} \subseteq \mathscr{C}oh_n$ and geometric points $\mathcal{E}_1,\mathcal{E}_2 \in \mathscr{U}(\kappa)$ that are identified in the adequate moduli space without admitting opposite filtrations. 

For example, in the open substack $\mathscr{B}un_2^{ss}$ of semistable rank $2$ vector bundles, any non-split extension $0 \to \mathcal{L} \to \mathcal{E} \to \mathcal{L}' \to 0$ is semistable and defines a point of $\mathscr{B}un_2^{ss}$, where $\mathcal{L} \ncong \mathcal{L}'$ are degree $0$ line bundles. These extensions are identified in the adequate moduli space of $\mathscr{B}un_2^{ss}$ without admitting opposite filtrations. 
\end{itemize}
\end{rmk}
It is known from Corollary \ref{cor:main} that no points get identified under the good moduli spaces $\mathscr{B}un_n^{simple} \to M_n^{simple}$ and $\mathscr{B}un_n^{s} \to M_n^{s}$. This can be slightly generalized.
\begin{lem}
Let $\mathscr{U} \subseteq \mathscr{B}un_n$ be an open substack. Then the following are equivalent:
\begin{enumerate}
\item 
No geometric points of $\mathscr{U}$ are quasi-S-equivalent.
\item 
The open substack $\mathscr{U}$ is contained in the open substack $\mathscr{B}un_n^{indec} \subseteq \mathscr{B}un_n$ of indecomposable vector bundles.
\end{enumerate}
Moreover, if $\mathscr{U}$ admits an adequate moduli space $\mathscr{U} \to U$, then (1) and (2) are further equivalent to:
\begin{enumerate}
\item[(3)]
No geometric points of $\mathscr{U}$ are identified in $U$.
\end{enumerate}
\end{lem}
\begin{proof}
If $\mathscr{U} \subseteq \mathscr{B}un_n^{indec}$, then it has no quasi-S-equivalent points by Lemma \ref{lem:S-in-Coh}. This shows $(2) \Rightarrow (1)$. Conversely, if $\mathscr{U}$ contains a decomposable vector bundle, say $\mathcal{E}_1 \oplus \mathcal{E}_2 \in \mathscr{U}(\kappa)$, then any extension of $\mathcal{E}_i$ by $\mathcal{E}_j$ lies in $\mathscr{U}$ since $\mathscr{U} \subseteq \mathscr{B}un_n$ is open, and all of them are quasi-S-equivalent by Lemma \ref{lem:S-in-Coh}.
\end{proof}
\subsection{Application to moduli of semistable bundles}
We first show that the identification of semistable vector bundles gives rise to the notion of S-equivalence.
\begin{cor}\label{cor:usual-S}
Two geometric points $\mathcal{E}_1,\mathcal{E}_2 \in \mathscr{B}un_n^{ss}(\kappa)$ are identified in the adequate moduli space if and only if they are S-equivalent.
\end{cor}
\begin{proof}
This follows from Corollary \ref{cor:main} and Lemma \ref{lem:qS-recover}.
\end{proof}
\subsection{Application to other moduli of bundles}
Let $\lambda$ be the polygon in the rank-degree plane consisting of vertices $\{(0,0),(1,1),(3,2)\}$ and denote by $\mathscr{B}un_3^{2,\leq \lambda} \subseteq \mathscr{B}un_3^2$ the open substack consisting of vector bundles whose Harder-Narasimhan polygons lie below $\lambda$. Let $\mathscr{U} \subseteq \mathscr{B}un_3^2$ be the open substack constructed in \cite[\S 5]{dario-zhang}, i.e.,
\[
\mathscr{U}=\left\langle \mathcal{E} \in \mathscr{B}un_3^{2,\leq \lambda}:
\begin{matrix}
\text{If } \mathcal{E} \text{ fits into } 0 \to \mathcal{L} \to \mathcal{E} \to \mathcal{F} \to 0 \text{ or } 0 \to \mathcal{F} \to \mathcal{E} \to \mathcal{L} \to 0 \\
\text{for some } \mathcal{L} \in \mathscr{B}un_1^1 \text{ and } \mathcal{F} \in \mathscr{B}un_2^1, \text{ then } \mathcal{F} \text{ is } (1,0)\text{-stable.}
\end{matrix}
\right\rangle.
\]
It contains unstable vector bundles and admits a separated non-proper good moduli space by \cite[Theorem 5.5]{dario-zhang}. By Corollary \ref{cor:main} two geometric points $\mathcal{E}_1, \mathcal{E}_2 \in \mathscr{U}(\kappa)$ are identified in the good moduli space if and only if there exist filtrations $\mathcal{E}_1^\bullet,\mathcal{E}_2^\bullet$ of $\mathcal{E}_1,\mathcal{E}_2$ such that $\mathrm{gr}(\mathcal{E}_1^\bullet)=\mathrm{gr}(\mathcal{E}_2^\bullet) \in \mathscr{U}(\kappa)$. By \cite[Lemma 5.3]{dario-zhang} we have
\[
\mathrm{gr}(\mathcal{E}_1^\bullet)=\mathrm{gr}(\mathcal{E}_2^\bullet)=\mathcal{L} \oplus \mathcal{F} \text{ for some } \mathcal{L} \in \mathscr{B}un_1^1(\kappa) \text{ and } \mathcal{F} \in \mathscr{B}un_2^{1,s}(1,0)(\kappa),
\]
i.e., both $\mathcal{E}_1$ and $\mathcal{E}_2$ are extensions of $\mathcal{L}$ and $\mathcal{F}$, one way or the other. Conversely, any extension of such $\mathcal{L}$ and $\mathcal{F}$ defines a point in $\mathscr{U}$. Therefore two geometric points $\mathcal{E}_1, \mathcal{E}_2 \in \mathscr{U}(\kappa)$ are identified in the good moduli space if and only if there exist $\mathcal{L} \in \mathscr{B}un_1^1(\kappa)$ and $\mathcal{F} \in \mathscr{B}un_2^{1,s}(1,0)(\kappa)$ such that
\[
\mathcal{E}_i \in \mathrm{Ext}^1(\mathcal{L},\mathcal{F}) \text{ or } \mathcal{E}_i \in \mathrm{Ext}^1(\mathcal{F},\mathcal{L}).
\]
In either case, the points $\mathcal{E}_1,\mathcal{E}_2$ are identified with the closed point $\mathcal{L} \oplus \mathcal{F}$ in the good moduli space. In particular, we see that every unstable vector bundle in $\mathscr{U}$ is identified with a semistable vector bundle. 
\begin{defn}
Let $\mathscr{U} \subseteq \mathscr{B}un_n$ be an open substack that admits an adequate moduli space. An unstable vector bundle in $\mathscr{U}$ is called \emph{essential} if it is not identified with a semistable vector bundle in the adequate moduli space of $\mathscr{U}$.
\end{defn}
\begin{example}
There are examples and non-examples on both sides.
\begin{enumerate}
\item 
The open substack $\mathscr{U} \subseteq \mathscr{B}un_3^2$ constructed above admits a separated good moduli space and no unstable vector bundle in $\mathscr{U}$ is essential.
\item 
The open substack $\mathscr{B}un_n^{simple} \subseteq \mathscr{B}un_n$ of simple vector bundles admits a non-separated good moduli space (see \cite[Corollary 3.19]{dario-zhang}) and trivially every unstable vector bundle in $\mathscr{B}un_n^{simple}$ is essential.
\end{enumerate}
\end{example}
\begin{question}
We are wondering if there exists an open substack $\mathscr{U} \subseteq \mathscr{B}un_n$ containing essential unstable vector bundles that admits a proper adequate moduli space $\mathscr{U} \to U$. The answer is no if $\mathscr{B}un_n^{ss} \subseteq \mathscr{U}$. Indeed, let $f: M_n^{ss} \to U$ be the morphism induced by the universal property of adequate moduli spaces (see \cite[Theorem 3.12]{Alper2019-1})
\[
\begin{tikzcd}
\mathscr{B}un_n^{ss} \ar[r,hook,"\circ" marking] \ar[d,"\mathrm{ams}"'] & \mathscr{U} \ar[d,"\mathrm{ams}"] \\
M_n^{ss} \ar[r,"f"'] & U.
\end{tikzcd}
\]
The morphism $f$ is dominant as $\mathscr{B}un_n^{ss} \subseteq \mathscr{U}$ is open dense and the vertical morphisms are surjective by \cite[Theorem 5.3.1 (1)]{MR3272912}. Since $M_n^{ss}$ is proper and $U$ is separated, the image of $M_n^{ss}$ is closed in $U$ by \cite[\href{https://stacks.math.columbia.edu/tag/04NX}{Tag 04NX}]{stacks-project}. This shows that $f$ is surjective. A diagram-chasing implies that every point of $\mathscr{U}-\mathscr{B}un_n^{ss}$, i.e., every unstable vector bundle in $\mathscr{U}$, is identified with a point in $\mathscr{B}un_n^{ss}$ under the adequate moduli space $\mathscr{U} \to U$. 
\end{question}
\section{Relation to separatedness}
Given an algebraic stack, it happens quite often that there exist non-isomorphic families with isomorphic generic fibers. For this stack to admit a separated moduli space, the special fibers of those families must be identified. We make this phenomenon into a definition.
\begin{defn}\label{defn:non-sep}
Let $\mathscr{X}$ be an algebraic stack over a field $k$. Two geometric points $x_1 \ncong x_2 \in \mathscr{X}(\kappa)$ are said to be \emph{non-separated} if there exist
\begin{itemize}
\item 
a DVR $R$ with fraction field $K$ and residue field $\kappa$, and
\item 
two families $x_{1,R},x_{2,R} \in \mathscr{X}(R)$ such that $x_{1,K} \cong x_{2,K}$ and $x_{i,\kappa} \cong x_i$.
\end{itemize}
\end{defn}
\begin{rmk}
The notion of non-separation comes from the test space for separatedness (see \cite[2.A and 2.B]{MR3758902}), which we recall here. For any DVR $R$ with fraction field $K$, residue field $\kappa$ and uniformizer $\pi \in R$, the test space for separatedness is a quotient stack:
\[
\overline{\mathrm{ST}}_R:=[\mathrm{Spec}(R[x,y]/xy-\pi)/\mathbf{G}_m],
\]
where $x,y$ have $\mathbf{G}_m$-weights $1,-1$ respectively. Denote by $0:=\mathrm{B}\mathbf{G}_{m,\kappa} \in \overline{\mathrm{ST}}_R$ its unique closed point defined by the vanishing of both $x$ and $y$. In this terminology non-separated points are nothing but special fibers of a non-trivial family
\[
\mathrm{Spec}(R) \cup_{\mathrm{Spec}(K)} \mathrm{Spec}(R)=\overline{\mathrm{ST}}_R-\{0\} \to \mathscr{X}.
\]
This notion can be seen as an algebraic analogue of non-Hausdorff separation in the topological sense, that is, for any non-separated geometric points $x_1,x_2 \in \mathscr{X}(\kappa)$, any of their (Zariski) open neighbourhoods intersect. 

Indeed, if the non-separation of $x_1,x_2$ is realized by $x_{1,R},x_{2,R} \in \mathscr{X}(R)$, then for any open neighbourhood $\mathscr{U}_i \subseteq \mathscr{X}$ of $x_i$ we have $x_{i,R} \in \mathscr{U}_i(R)$ since $\mathscr{U}_i \subseteq \mathscr{X}$ is open. The assumption $x_{1,K} \cong x_{2,K} \in \mathscr{U}_1 \cap \mathscr{U}_2(K)$ implies that $\mathscr{U}_1 \cap \mathscr{U}_2 \neq \emptyset$. In this way we recover the usual non-Hausdorff separation.
\end{rmk}
\begin{lem}\label{lem:non-imply-qS}
Non-separated geometric points in an S-complete algebraic stack are quasi-S-equivalent.
\end{lem}
Therefore, non-separated points in a locally reductive algebraic stack with a separated adequate moduli space are identified by Theorem \ref{thm:equivalence}.
\begin{proof}
Let $\mathscr{X}$ be an S-complete algebraic stack and $x_1,x_2 \in \mathscr{X}(\kappa)$ be non-separated geometric points realized by two families $x_{1,R},x_{2,R} \in \mathscr{X}(R)$. These two families define a morphism $\overline{\mathrm{ST}}_R-\{0\} \to \mathscr{X}$ and it extends to $h: \overline{\mathrm{ST}}_R \to \mathscr{X}$ since $\mathscr{X}$ is S-complete. The quasi-S-equivalence of $x_1,x_2$ is realized by the restrictions $h|_{x=0}: [\mathrm{Spec}(\kappa[y])/\mathbf{G}_{m,\kappa}] \to \mathscr{X}$ and $h|_{y=0}: [\mathrm{Spec}(\kappa[x])/\mathbf{G}_{m,\kappa}] \to \mathscr{X}$.
\end{proof}
Together with Lemma \ref{lem:qS-recover} this shows
\begin{cor}
Non-separated geometric points in $\mathscr{B}un_n^{ss}$ are S-equivalent.
\end{cor}
The converse is not true, i.e., not all S-equivalent semistable vector bundles are non-separated, they need have opposite filtrations, see Corollary \ref{cor:nonsep=id} (1) later.

Here are some basic properties of non-separated coherent sheaves, which can be compared with \cite[Lemma 3.16]{dario-zhang}.
\begin{lem}\label{lem:non-separated}
For any non-separated geometric points $\mathcal{E}_1 \ncong \mathcal{E}_2 \in \mathscr{C}oh_n(\kappa)$, there exist
\begin{enumerate}
\item
an isomorphism $\det(\mathcal{E}_1) \cong \det(\mathcal{E}_2)$.
\item
non-zero morphisms $h_{ij}: \mathcal{E}_i \to \mathcal{E}_j$ such that $h_{ij} \circ h_{ji}=0$.
\end{enumerate}
\end{lem}
\begin{proof}
Suppose the non-separation of $\mathcal{E}_1, \mathcal{E}_2$ is realized by $\mathcal{G}_{1,R},\mathcal{G}_{2,R} \in \mathscr{C}oh_n(R)$ for some DVR $R$ with fraction field $K$ and residue field $\kappa$.

Since the Picard variety $\mathrm{Pic}$ is separated, $\det(\mathcal{G}_{1,K}) \cong \det(\mathcal{G}_{2,K})$ implies that $\det(\mathcal{G}_{1,R}) \cong \det(\mathcal{G}_{2,R})$. In particular, $\det(\mathcal{G}_{1,\kappa}) \cong \det(\mathcal{G}_{2,\kappa})$, i.e., $\det(\mathcal{E}_1) \cong \det(\mathcal{E}_2)$. This proves (1). For (2) we fix an isomorphism $\lambda_K: \mathcal{G}_{1,K} \xrightarrow{\sim} \mathcal{G}_{2,K}$ and use it to identify $\mathcal{G}_{1,K}$ and $\mathcal{G}_{2,K}$. Regard both $\mathcal{G}_{1,R}$ and $\mathcal{G}_{2,R}$ as subsheaves of $\mathcal{G}_{1,K} \cong \mathcal{G}_{2,K}$. Let $\pi \in R$ be a uniformizer. Let $a \in \mathbf{N}$ (resp., $b \in \mathbf{N}$) be the minimal integer such that $\mathcal{G}_{1,R} \subseteq \pi^{-a}\mathcal{G}_{2,R}$ (resp., $\mathcal{G}_{2,R} \subseteq \pi^{-b}\mathcal{G}_{1,R}$). Let $h_{12,R}: \mathcal{G}_{1,R} \to \mathcal{G}_{2,R}$ be the following composition
\[
\begin{tikzcd}
\mathcal{G}_{1,R} \ar[d,equal] \\
\mathcal{G}_{1,R} \cap \pi^{-a}\mathcal{G}_{2,R} \ar[r,"\pi"] & \mathcal{G}_{1,R} \cap \pi^{-(a-1)}\mathcal{G}_{2,R} \ar[r,"\pi"] & \cdots \ar[r,"\pi"] & \mathcal{G}_{1,R} \cap \pi^{-1}\mathcal{G}_{2,R} \ar[d,"\pi"] \\
& & & \mathcal{G}_{1,R} \cap \mathcal{G}_{2,R} \ar[d,"\text{icl}"] \\
\pi^{-b}\mathcal{G}_{1,R} \cap \mathcal{G}_{2,R} \ar[d,equal] & \pi^{-(b-1)}\mathcal{G}_{1,R} \cap \mathcal{G}_{2,R} \ar[l,"\text{icl}"] & \cdots \ar[l,"\text{icl}"] & \pi^{-1}\mathcal{G}_{1,R} \cap \mathcal{G}_{2,R} \ar[l,"\text{icl}"] \\
\mathcal{G}_{2,R}
\end{tikzcd}
\]
Reversing all arrows in this composition and interchanging $\pi$ and icl gives an inverse morphism $h_{21,R}: \mathcal{G}_{2,R} \to \mathcal{G}_{1,R}$ such that 
\[
h_{ij,R} \circ h_{ji,R}=\pi^{a+b}.
\]
The condition $\mathcal{E}_1 \ncong \mathcal{E}_2$ then implies that $a+b>0$, so
\[
h_{ij,\kappa} \circ h_{ji,\kappa}=0.
\]
Both morphisms $h_{ij,\kappa}$ are non-zero because $a$ and $b$ are also the minimal integers such that $\pi^a\lambda_K$ and $\pi^b\lambda_K^{-1}$ define morphisms $\mathcal{G}_{1,R} \to \mathcal{G}_{2,R}$ and $\mathcal{G}_{2,R} \to \mathcal{G}_{1,R}$ extending $\lambda_K$ and $\lambda_K^{-1}$ respectively.
\end{proof}
\begin{rmk}
If $k=\mathbf{C}$ and $\mathcal{E}_1 \ncong \mathcal{E}_2 \in \mathscr{B}un_n(\kappa)$, then Lemma \ref{lem:non-separated} (2) is proved in \cite[Proposition 2]{MR529671} using analytic methods. Here we give an alternative argument, basically coming from \cite[Proof of Proposition 6.5]{MR3013028}, to show that there exists a non-zero morphism $h_{ij}: \mathcal{E}_i \to \mathcal{E}_j$ for $i \neq j$. Suppose the non-separation of $\mathcal{E}_1, \mathcal{E}_2$ is realized by $\mathcal{G}_{1,R},\mathcal{G}_{2,R} \in \mathscr{B}un_n(R)$ for some DVR $R$ with fraction field $K$ and residue field $\kappa$. Consider the Cartesian diagram
\[
\begin{tikzcd}
C \times \mathrm{Spec}(R) \ar[r,"q"] \ar[d,"p"'] & C \times \kappa \ar[d] \\
\mathrm{Spec}(R) \ar[r] & \mathrm{Spec}(\kappa) \arrow[ul,phantom,"\lrcorner"],
\end{tikzcd}
\]
and the coherent sheaf $\mathbf{R}^1p_*(\mathcal{G}_{i,R} \otimes \mathcal{G}_{j,R}^\vee \otimes q^*\omega_C)$ over $\mathrm{Spec}(R)$. Its generic fiber
\[
H^1(C_K,\mathcal{G}_{i,K} \otimes \mathcal{G}_{j,K}^\vee \otimes q_K^*\omega_C) \cong H^0(C_K,\mathcal{G}_{i,K}^\vee \otimes \mathcal{G}_{j,K})^\vee \cong \mathrm{Hom}(\mathcal{G}_{i,K},\mathcal{G}_{j,K})^\vee
\]
is non-zero, so is its special fiber by upper semi-continuity. Then $\mathrm{Hom}(\mathcal{E}_i,\mathcal{E}_j)^\vee \cong \mathrm{Hom}(\mathcal{G}_{i,\kappa},\mathcal{G}_{j,\kappa})^\vee \neq 0$. Thus there exists a non-zero morphism $\mathcal{E}_i \to \mathcal{E}_j$.
\end{rmk}
\begin{example}
Non-separation phenomena do not occur in $\mathscr{B}un_n^{s}$. There are several ways to see this:
\begin{itemize}
\item
For any geometric points $\mathcal{E}_1 \ncong \mathcal{E}_2 \in \mathscr{B}un_n^{s}(\kappa)$, we have $\mathrm{Hom}(\mathcal{E}_i,\mathcal{E}_j)=0$ and can apply Lemma \ref{lem:non-separated} (2).
\item
Any families $\mathcal{G}_{1,R},\mathcal{G}_{2,R} \in \mathscr{B}un_n^{s}(R)$ coinciding on the generic fiber must be isomorphic on the special fiber by Langton's theorem A (see \cite{MR0364255}).
\end{itemize}
\end{example}
To describe non-separated points in arbitrary algebraic stacks we generalize the notion of opposite filtrations introduced in \cite[Definition 3.15]{dario-zhang}.
\begin{defn}\label{defn:opp-fil}
Let $\mathscr{X}$ be an algebraic stack over a field $k$. Two geometric points $x_1,x_2 \in \mathscr{X}(\kappa)$ are said to have \emph{opposite filtrations} if there exists a morphism
\[
f: [\mathrm{Spec}(\kappa[x,y]/xy)/\mathbf{G}_{m,\kappa}] \to \mathscr{X},
\]
where $x,y$ have $\mathbf{G}_m$-weights $1,-1$ respectively, 
such that 
\[
f(1,0) \cong x_1 \text{ and } f(0,1) \cong x_2.
\]
\end{defn}
\begin{example}[Opposite filtrations for coherent sheaves]\label{eg:opp-in-Coh}
Let $\mathscr{U} \subseteq \mathscr{C}oh_n$ be an open substack. By restricting to the irreducible components $x=0$ and $y=0$ respectively and applying the Rees construction, a morphism $[\mathrm{Spec}(\kappa[x,y]/xy)/\mathbf{G}_{m,\kappa}] \to \mathscr{U}$ is the data of two geometric points $\mathcal{E},\mathcal{E}' \in \mathscr{U}(\kappa)$ together with two finite $\mathbf{Z}$-graded filtrations $\mathcal{E}^\bullet,\mathcal{E}'_\bullet$ of subsheaves (which are called opposite filtrations in \cite[Definition 3.15]{dario-zhang})
\begin{gather*}
\mathcal{E}^\bullet: 0 \subseteq \cdots \subseteq \mathcal{E}^{i-1} \subseteq \mathcal{E}^{i} \subseteq \mathcal{E}^{i+1} \subseteq \cdots \subseteq \mathcal{E} \\ 
\mathcal{E}'_\bullet: \mathcal{E}' \supseteq \cdots \supseteq \mathcal{E}'_{i-1} \supseteq \mathcal{E}'_{i} \supseteq \mathcal{E}'_{i+1} \supseteq \cdots \supseteq 0
\end{gather*}
such that $\mathcal{E}^i/\mathcal{E}^{i-1} \cong \mathcal{E}'_{i}/\mathcal{E}'_{i+1}$ for all $i \in \mathbf{Z}$, and $\mathrm{gr}(\mathcal{E}^\bullet)=\mathrm{gr}(\mathcal{E}'_\bullet) \in \mathscr{U}(\kappa)$ (this condition is superfluous if $\mathscr{U}=\mathscr{C}oh_n$).
\end{example}
It turns out that, for any geometric points in an S-complete algebraic stacks, having opposite filtrations is equivalent to being non-separated. For this we need the following extension result, whose proof employs deformation theory and the powerful coherent completeness.
\begin{lem}\label{lem:res-equiv}
Let $\mathscr{X}$ be a smooth algebraic stack over a field $k$. If $\mathscr{X}$ either has affine stabilizers, or has quasi-affine diagonal, or is Deligne-Mumford, then for any field $\kappa$, the restriction morphism
\[
\mathrm{res}: \mathrm{Map}(\overline{\mathrm{ST}}_{R},\mathscr{X}) \xrightarrow{\pi=0} \mathrm{Map}([\mathrm{Spec}(\kappa[x,y]/xy)/\mathbf{G}_{m,\kappa}],\mathscr{X}).
\]
is surjective, where $R:=\kappa[[\pi]]$.
\end{lem}
\begin{proof}
Let $A:=R[x,y]/xy-\pi$. It is a complete noetherian ring with respect to the $\pi$-adic topology. For any morphism
\[
f_1: [\mathrm{Spec}(\kappa[x,y]/xy)/\mathbf{G}_{m,\kappa}]=[\mathrm{Spec}(A/\pi)/\mathbf{G}_{m,\kappa}] \to \mathscr{X},
\]
we want to construct a morphism $f: \overline{\mathrm{ST}}_{R} \to \mathscr{X}$ such that $f|_{xy=0}=f_1$. Since $\mathscr{X}$ is smooth, any morphism $\mathrm{Spec}(A/\pi) \to \mathscr{X}$ can be lifted infinitesimally to $\mathrm{Spec}(A/\pi^i)$ for all $i>1$. The obstruction to the existence of a $\mathbf{G}_m$-equivariant lifting lies in
\[
H^1\left(\mathrm{B}\mathbf{G}_{m,\kappa},f_{i-1}^*\mathbf{T}_{\mathscr{X}}|_{\mathrm{B}\mathbf{G}_{m,\kappa}} \otimes (\pi^{i-1})/(\pi^i)\right)=0.
\]
Thus there exists a compatible family of morphisms $f_i: [\mathrm{Spec}(A/\pi^i)/\mathbf{G}_{m,\kappa}] \to \mathscr{X}$. By the assumption on $\mathscr{X}$, the coherent completeness (see \cite[Corollary 1.5]{MR4009673}) yields a morphism 
\[
f: [\mathrm{Spec}(A)/\mathbf{G}_{m,\kappa}]=\overline{\mathrm{ST}}_R \to \mathscr{X} 
\]
and $f|_{xy=0}=f_1$. This shows the surjectivity.
\end{proof}
\begin{thm}\label{thm:equivalence-non-opp}
Let $\mathscr{X}$ be a smooth S-complete algebraic stack over a field $k$. Suppose $\mathscr{X}$ either has affine stabilizers, or has quasi-affine diagonal, or is Deligne-Mumford. Then for any geometric points $x_1,x_2 \in \mathscr{X}(\kappa)$, the following are equivalent:
\begin{enumerate}
\item
The points $x_1,x_2$ are non-separated.
\item
The points $x_1,x_2$ have opposite filtrations.
\end{enumerate}
\end{thm}
\begin{proof}
(1) $\Rightarrow$ (2) (using S-completeness only): Keep the notations in the proof of Lemma \ref{lem:non-imply-qS}. The restriction $h|_{xy=0}: [\mathrm{Spec}(\kappa[x,y]/xy)/\mathbf{G}_{m,\kappa}] \to \mathscr{X}$ gives rise to opposite filtrations of $x_1,x_2$.

(2) $\Rightarrow$ (1) (without using S-completeness): For any morphism 
\[
f: [\mathrm{Spec}(\kappa[x,y]/xy)/\mathbf{G}_{m,\kappa}] \to \mathscr{X}
\]
such that $f^\circ(1,0) \cong x_1$ and $f^\circ(0,1) \cong x_2$, let $\hat{f}: \overline{\mathrm{ST}}_R \to \mathscr{X}$ be an extension of $f$ by Lemma \ref{lem:res-equiv}, where $R:=\kappa[[\pi]]$. The restrictions $\hat{f}|_{x \neq 0},\hat{f}|_{y \neq 0}: \mathrm{Spec}(R) \to \mathscr{X}$ realize the non-separation of $x_1,x_2$.
\end{proof}
This reproves \cite[Corollary 3.23]{dario-zhang}.
\begin{cor}\label{cor:non=opp-in-Coh}
Two geometric points in $\mathscr{C}oh_n$ (or $\mathscr{B}un_n^{ss}$) are non-separated if and only if they have opposite filtrations.
\end{cor}
\begin{rmk}
There is a bundle-theoretical construction of opposite filtrations for non-separated points in $\mathscr{B}un_n^{ss}$, using non-zero morphisms between them. 

Indeed, for any non-separated  points $\mathcal{E}_1 \ncong \mathcal{E}_2 \in \mathscr{B}un_n^{ss}(\kappa)$, let $h_{ij}: \mathcal{E}_i \to \mathcal{E}_j$ be a non-zero morphism such that $h_{ij} \circ h_{ji}=0$ (see Lemma \ref{lem:non-separated} (2)). By the semistability of $\mathcal{E}_1,\mathcal{E}_2$ we have a commutative diagram with exact rows
\[
\begin{tikzcd}
0 \ar[r] & \ker(h_{12}) \ar[r] & \mathcal{E}_1 \ar[d,"h_{12}",shift left] \ar[r] & \mathrm{Im}(h_{12}) \ar[r] \ar[d,"\cong"] & 0 \\
0 & \mathrm{Im}(h_{21}) \ar[l] \ar[u,"\cong"] & \mathcal{E}_2 \ar[l] \ar[u,shift left,"h_{21}"] & \ker(h_{21}) \ar[l] & 0. \ar[l]
\end{tikzcd}
\]
This shows that $\mathcal{E}_1,\mathcal{E}_2$ have opposite filtrations by Example \ref{eg:opp-in-Coh}. This is exactly how Langton showed in \cite{MR0364255} that semistable limits are S-equivalent.
\end{rmk}
It has been observed in \cite[Remark 3.40]{MR4665776} that S-completeness need not be preserved when passing to open substacks. However, there is a sufficient and necessary condition for open substacks of an S-complete algebraic stack to be S-complete.
\begin{defn}[Axis-complete]\label{defn:axis-comp}
A morphism $f: \mathscr{X} \to \mathscr{Y}$ of locally noetherian algebraic
stacks over a field $k$ is \emph{axis-complete} if for any field $\kappa$ and any commutative diagram
\[
\begin{tikzcd}
\left[\mathrm{Spec}(\kappa[x,y]/xy)/\mathbf{G}_{m,\kappa}\right]-\{0\} \ar[r] \ar[d,hook] & \mathscr{X} \ar[d,"f"] \\
\left[\mathrm{Spec}(\kappa[x,y]/xy)/\mathbf{G}_{m,\kappa}\right] \ar[r] \ar[ur,dashed] & \mathscr{Y}
\end{tikzcd} 
\]
of solid arrows, there exists a unique dotted arrow filling in.
\end{defn}
For any algebraic stack $\mathscr{X}$ we have a diagram
\[
\begin{tikzcd}
\mathrm{Map}(\overline{\mathrm{ST}}_R,\mathscr{X}) \ar[r,"\mathrm{res}"] \ar[d,"\mathrm{res}"'] & \mathrm{Map}([\mathrm{Spec}(\kappa[x,y]/xy)/\mathbf{G}_m],\mathscr{X}) \ar[d,"\mathrm{res}"] \\
\mathrm{Map}(\overline{\mathrm{ST}}_R-\{0\},\mathscr{X}) \ar[r,dashed,"?"] & \mathrm{Map}([\mathrm{Spec}(\kappa[x,y]/xy)/\mathbf{G}_m]-\{0\},\mathscr{X}).
\end{tikzcd}
\]
One obstruction to relate S-completeness (LHS) and axis-completeness (RHS) is that we do not have a ``restriction'' morphism downstairs. If the left vertical arrow has a section (e.g. if $\mathscr{X}$ is S-complete), then we can use this to define a morphism downstairs. This is also the reason why we need the S-completeness assumption in Theorem \ref{thm:equivalence-non-opp}.
\begin{thm}\label{thm:open-S-comp}
Let $\mathscr{X}$ be a smooth S-complete algebraic stack over a field $k$. Suppose $\mathscr{X}$ either has affine stabilizers, or has quasi-affine diagonal, or is Deligne-Mumford. Then an open substack $\mathscr{U} \subseteq \mathscr{X}$ is S-complete if and only if the open immersion $\mathscr{U} \hookrightarrow \mathscr{X}$ is axis-complete (e.g. if $\mathscr{U}$ is axis-complete). 
\end{thm}
\begin{proof}
Suppose the open immersion $\mathscr{U} \hookrightarrow \mathscr{X}$ is axis-complete. For any DVR $R$ with residue field $\kappa$ and any morphism $\overline{\mathrm{ST}}_R-\{0\} \to \mathscr{U}$, we can complete it in $\mathscr{X}$ as $\mathscr{X}$ is S-complete:
\[
\begin{tikzcd}
\overline{\mathrm{ST}}_R-\{0\} \ar[r] \ar[d,hook] \ar[dr,"\circ"] & \mathscr{U} \ar[d,hook] \\
\overline{\mathrm{ST}}_R \ar[r,"\exists ! \hat{f}"'] & \mathscr{X}.
\end{tikzcd}
\]
It suffices to show $\hat{f}(0) \in \mathscr{U}(\kappa)$. Note that the restriction 
\[
f:=\hat{f}|_{xy=0}: [\mathrm{Spec}(\kappa[x,y]/xy)/\mathbf{G}_{m,\kappa}] \to \mathscr{X}
\]
maps $[\mathrm{Spec}(\kappa[x,y]/xy)/\mathbf{G}_{m,\kappa}]-\{0\}$ to $\mathscr{U}$, so $\hat{f}(0)=f(0) \in \mathscr{U}(\kappa)$ as the open immersion $\mathscr{U} \hookrightarrow \mathscr{X}$ is axis-complete.

Suppose $\mathscr{U}$ is S-complete. For any field $\kappa$ and any commutative diagram 
\[
\begin{tikzcd}
\left[\mathrm{Spec}(\kappa[x,y]/xy)/\mathbf{G}_{m,\kappa}\right]-\{0\} \ar[r] \ar[d,hook] & \mathscr{U} \ar[d,hook] \\
\left[\mathrm{Spec}(\kappa[x,y]/xy)/\mathbf{G}_{m,\kappa}\right] \ar[r,"f"'] & \mathscr{X}
\end{tikzcd} 
\]
of solid arrows, it suffices to show $f(0) \in \mathscr{U}(\kappa)$. By Lemma \ref{lem:res-equiv} there exists a morphism $\hat{f}: \overline{\mathrm{ST}}_R \to \mathscr{X}$ with $R:=\kappa[[\pi]]$ such that $f=\hat{f}|_{xy=0}$, and $\hat{f}$ maps $\overline{\mathrm{ST}}_R-\{0\}$ to $\mathscr{U}$ since $\mathscr{U} \subseteq \mathscr{X}$ is open. Then $f(0)=\hat{f}(0) \in \mathscr{U}(\kappa)$ as $\mathscr{U}$ is S-complete.
\end{proof}
\begin{lem}\label{lem:open-axis-complete}
An open immersion $\mathscr{U} \hookrightarrow \mathscr{C}oh_n$ is axis-complete if and only if for any geometric points $\mathcal{E}, \mathcal{E}' \in \mathscr{U}(\kappa)$ with opposite filtrations $\mathcal{E}^\bullet,\mathcal{E}'_\bullet$, we have $\mathrm{gr}(\mathcal{E}^\bullet)=\mathrm{gr}(\mathcal{E}'_\bullet) \in \mathscr{U}(\kappa)$.
\end{lem}
\begin{proof}
For the only if part: by Example \ref{eg:opp-in-Coh}, any geometric points $\mathcal{E},\mathcal{E}' \in \mathscr{U}(\kappa)$ with opposite filtrations $\mathcal{E}^\bullet,\mathcal{E}'_\bullet$ define a morphism $f: [\mathrm{Spec}(\kappa[x,y]/xy)/\mathbf{G}_{m,\kappa}] \to \mathscr{C}oh_n$ and $f$ maps $[\mathrm{Spec}(\kappa[x,y]/xy)/\mathbf{G}_{m,\kappa}]-\{0\}$ to $\mathscr{U}$ since $\mathscr{U} \subseteq \mathscr{C}oh_n$ is open. Then $f(0)=\mathrm{gr}(\mathcal{E}^\bullet)=\mathrm{gr}(\mathcal{E}'_\bullet) \in \mathscr{U}(\kappa)$ as the open immersion $\mathscr{U} \hookrightarrow \mathscr{C}oh_n$ is axis-complete.

For the if part: for any field $\kappa$ and any commutative diagram
\[
\begin{tikzcd}
\left[\mathrm{Spec}(\kappa[x,y]/xy)/\mathbf{G}_{m,\kappa}\right]-\{0\} \ar[r] \ar[d,hook] & \mathscr{U} \ar[d,hook] \\
\left[\mathrm{Spec}(\kappa[x,y]/xy)/\mathbf{G}_{m,\kappa}\right] \ar[r,"f"'] & \mathscr{C}oh_n
\end{tikzcd} 
\]
of solid arrows, it suffices to show $f(0) \in \mathscr{U}(\kappa)$. By Example \ref{eg:opp-in-Coh}, the morphism $f$ corresponds to two geometric points $\mathcal{E},\mathcal{E}' \in \mathscr{U}(\kappa)$ together with opposite filtrations $\mathcal{E}^\bullet,\mathcal{E}'_\bullet$, and $f(0)=\mathrm{gr}(\mathcal{E}^\bullet)=\mathrm{gr}(\mathcal{E}'_\bullet)$. This is in $\mathscr{U}(\kappa)$ by assumption.
\end{proof}
This reproves \cite[Proposition 3.17]{dario-zhang}.
\begin{cor}\label{cor:nonsep=id}
An open substack $\mathscr{U}$ is S-complete if and only if for any geometric points $\mathcal{E}, \mathcal{E}' \in \mathscr{U}(\kappa)$ with opposite filtrations $\mathcal{E}^\bullet,\mathcal{E}'_\bullet$, we have $\mathrm{gr}(\mathcal{E}^\bullet)=\mathrm{gr}(\mathcal{E}'_\bullet) \in \mathscr{U}(\kappa)$.
\end{cor}

\bibliographystyle{plain}

\bibliography{S-equivalence}

\begin{thebibliography}{10}

\bibitem{MR3272912}
Jarod Alper.
\newblock Adequate moduli spaces and geometrically reductive group schemes.
\newblock {\em Algebr. Geom.}, 1(4):489--531, 2014.

\bibitem{MR4480534}
Jarod Alper, Pieter Belmans, Daniel Bragg, Jason Liang, and Tuomas Tajakka.
\newblock Projectivity of the moduli space of vector bundles on a curve.
\newblock In {\em Stacks {P}roject {E}xpository {C}ollection ({SPEC})}, volume
  480 of {\em London Math. Soc. Lecture Note Ser.}, pages 90--125. Cambridge
  Univ. Press, Cambridge, 2022.

\bibitem{Alper2019-1}
Jarod Alper, Jack Hall, and David Rydh.
\newblock The \'{e}tale local structure of algebraic stacks.
\newblock \href{https://arxiv.org/abs/1912.06162}{arXiv: 1912.06162}, December
  2019.

\bibitem{MR4665776}
Jarod Alper, Daniel Halpern-Leistner, and Jochen Heinloth.
\newblock Existence of moduli spaces for algebraic stacks.
\newblock {\em Invent. Math.}, 234(3):949--1038, 2023.

\bibitem{MR86358}
M.~Atiyah.
\newblock On the {K}rull-{S}chmidt theorem with application to sheaves.
\newblock {\em Bull. Soc. Math. France}, 84:307--317, 1956.

\bibitem{MR4009673}
Jack Hall and David Rydh.
\newblock Coherent {T}annaka duality and algebraicity of {H}om-stacks.
\newblock {\em Algebra Number Theory}, 13(7):1633--1675, 2019.

\bibitem{DHL2014}
Daniel Halpern-Leistner.
\newblock On the structure of instability in moduli theory.
\newblock \href{https://arxiv.org/abs/1411.0627}{arXiv: 1411.0627}, 2014.

\bibitem{MR3013028}
Georg Hein.
\newblock Faltings' construction of the moduli space of vector bundles on a
  smooth projective curve.
\newblock In {\em Affine flag manifolds and principal bundles}, Trends Math.,
  pages 91--122. Birkh\"{a}user/Springer Basel AG, Basel, 2010.

\bibitem{MR3758902}
Jochen Heinloth.
\newblock Hilbert-{M}umford stability on algebraic stacks and applications to
  {$\mathcal{G}$}-bundles on curves.
\newblock {\em \'{E}pijournal G\'{e}om. Alg\'{e}brique}, 1:Art. 11, 37, 2017.

\bibitem{MR2665168}
Daniel Huybrechts and Manfred Lehn.
\newblock {\em The geometry of moduli spaces of sheaves}.
\newblock Cambridge Mathematical Library. Cambridge University Press,
  Cambridge, second edition, 2010.

\bibitem{MR0364255}
Stacy~G. Langton.
\newblock Valuative criteria for families of vector bundles on algebraic
  varieties.
\newblock {\em Ann. of Math. (2)}, 101:88--110, 1975.

\bibitem{MR529671}
V.~Alan Norton.
\newblock Analytic moduli of complex vector bundles.
\newblock {\em Indiana Univ. Math. J.}, 28(3):365--387, 1979.

\bibitem{stacks-project}
The {Stacks Project Authors}.
\newblock Stacks project.
\newblock \url{https://stacks.math.columbia.edu}, 2024.

\bibitem{dario-zhang}
Dario Wei{\ss}mann and Xucheng Zhang.
\newblock A stacky approach to identifying the semistable locus of bundles.
\newblock \href{https://arxiv.org/pdf/2302.09245.pdf}{arXiv: 2302.09245}, 2024.
\newblock to appear in \textit{Algebraic Geometry}.

\end{thebibliography}

\end{document}